\documentclass{amsart}

\usepackage{mathptmx, amssymb}

\usepackage{color,xcolor}
\definecolor{chianti}{rgb}{0.6,0,0}
\definecolor{meretale}{rgb}{0,0,.6}
\definecolor{leaf}{rgb}{0,.35,0}
\usepackage[colorlinks=true, pagebackref, hyperindex, citecolor=meretale, urlcolor=leaf, linkcolor=chianti]{hyperref}

\usepackage{tikz}
\usetikzlibrary{calc}

\usepackage{doi}

\newtheorem{theorem}{Theorem}[section]
\newtheorem{lemma}[theorem]{Lemma}
\newtheorem{corollary}[theorem]{Corollary}
\newtheorem{proposition}[theorem]{Proposition}

\theoremstyle{definition}
\newtheorem{definition}[theorem]{Definition}
\newtheorem{example}[theorem]{Example}

\numberwithin{equation}{section}

\DeclareMathOperator{\Brad}{B-rad}

\def\GL{\operatorname{GL}}

\def\HS{\operatorname{HS}}

\def\gin{\operatorname{gin}}
\def\mm{\mathfrak{m}}
\newcommand{\M}{{\mathcal M}}
\DeclareMathOperator{\pol}{pol}

\def\ini{\operatorname{in}}

\def\to{\longrightarrow}

\def\mapsto{\longmapsto}

\def\x{\bf{x}}
\def\y{\bf{y}}
\def\z{\bf{z}}

\def\cA{\mathcal{A}}

\def\FF{\mathbb{F}}
\def\NN{\mathbb{N}}
\def\QQ{\mathbb{Q}}
\def\ZZ{\mathbb{Z}}
\def\fF{\mathfrak{F}}
\def\fG{\mathfrak{G}}

\def\cK{\mathcal{K}}
\newcommand{\nc}{\newcommand}

\nc{\Aldo}[1]{{\color{red} \sf $\heartsuit$ [#1]}}
\nc{\Emanuela}[1]{{\color{leaf} \sf $\heartsuit$ [#1]}}
\nc{\Volkmar}[1]{{\color{blue} \sf $\heartsuit$ [#1]}}
\nc{\red}{\textcolor{red}}
\nc{\green}{\textcolor{green}}
\nc{\blue}{\textcolor{blue}}

\begin{document}

\title{Generalized binomial edge ideals are Cartwright-Sturmfels}

\author{Aldo Conca}
\address{Dipartimento di Matematica, Universit\`a di Genova, Dipartimento di Eccellenza 2023-2027, Via Dodecaneso 35, I-16146 Genova, Italy}
\email{aldo.conca@unige.it}

\author{Emanuela De Negri}
\address{Dipartimento di Matematica, Universit\`a di Genova, Dipartimento di Eccellenza 2023-2027, Via Dodecaneso 35, I-16146 Genova, Italy}
\email{emanuela.denegri@unige.it}

\author{Volkmar Welker}
\address{Philipps Universit\"at Marbu\-rg, Fachbereich Mathematik und Informatik, 35032 Marburg, Germany }
\email{welker@mathematik.uni-marburg.de}

\thanks{A.C. and E.D.N. are supported by PRIN~2020355B8Y ``Squarefree Gr\"obner degenerations, special varieties and related topics,'' by MIUR Excellence Department Project awarded to the Department of Mathematics, University of Genova, CUP D33C23001110001 and by INdAM-GNSAGA}

\dedicatory{Dedicated to the Memory of our Teacher and Mentor J\"urgen Herzog }

\begin{abstract}
Binomial edge ideals associated to a simple graph $G$ were introduced in \cite{HHHKR} by Herzog and collaborators  and, independently, by Ohtani in \cite{Ohtani}. They became an ``instant classic" in combinatorial commutative algebra with more than 100 papers devoted to their investigation over the past 15 years. They exhibit many striking properties, including being radical and, moreover, Cartwright–Sturmfels. Using the fact that binomial edge ideals can be seen as ideals of $2$-minors of a matrix of variables with two rows, generalized binomial edge ideals of $2$-minors of matrices of $m$ rows were introduced by Rauh \cite{Rauh} and proved to be radical. The goal of this paper is to prove that generalized binomial edge ideals are Cartwright–Sturmfels. On the way we provide results on ideal constructions preserving the Cartwright-Sturmfels property. We also give examples and counterexamples to
the Cartwright-Sturmfels property for higher minors.
\end{abstract}
\maketitle

\section{Introduction}

Ideals with a radical initial ideal are special: their homological relationship with the initial ideal is exceptionally tight. This intuition underlies Herzog’s conjecture, now a theorem \cite{CoVa}, and it is precisely why such ideals are nowadays referred to as Herzog ideals \cite{HTVW}.
Cartwright–Sturmfels ideals, introduced in \cite{CDG2}, form a particularly strong subclass of Herzog ideals. A $\ZZ^n$-graded ideal in a $\ZZ^n$-graded polynomial ring is Cartwright-Sturmfels if its $\ZZ^n$-graded generic initial ideal is radical (this is one of the possible definitions). In fact, every initial ideal of a Cartwright-Sturmfels ideal is radical.
Over the last years, several families of Cartwright-Sturmfels ideals have been identified \cite{CDG1,CDG2,CDG3,CDG4,CDG5,CoWe,MaTa}, generalizations have been proposed \cite{CCC} and the Hilbert-schemes associated to them have been investigated \cite{RaSa}. 

Binomial edge ideals are generated by collections of $2$-minors of a generic $2\times n$ matrix that correspond to the edges of a graph. They were introduced by Herzog and collaborators in \cite{HHHKR} and, independently, by Ohtani in \cite{Ohtani}. The notion  was later generalized by Rauh in \cite{Rauh}. Ordinary binomial edge ideals were proved to be Cartwright-Sturmfels in \cite{CDG4,CDG5}.
The main objective of this paper is to show that generalized binomial edge ideals are also Cartwright–Sturmfels, see  Theorem \ref{main}.  We also prove two auxiliary results concerning general Cartwright-Sturmfels ideals. Firstly, we describe  certain  Cartwright-Sturmfels subideals of a given Cartwright–Sturmfels ideal (Proposition \ref{easy2}). Secondly, we show that the sum of two Cartwright–Sturmfels ideals is Cartwright–Sturmfels (Proposition \ref{prop:linked}) if the multidegrees of the generators of the two ideals overlap at most in one component. 
We conclude the paper with a discussion of families of higher-order minors that define Cartwright–Sturmfels ideals and families that do not define 
Cartwright-Sturmfels ideals. Our results in this direction are preliminary and leave room for further refinement.


\section{Cartwright-Sturmfels ideals }
\label{section:CS}

Let $K$ be a field and let
\(S=K[x_{ij}\mid 1\leq j\leq n,\ 1\leq i\leq m_j]\)
 be the polynomial ring over $K$ endowed with the $\ZZ^n$-grading induced by setting $\deg(x_{ij})=e_j$, where $e_j\in\ZZ^n$ is the $j$-th standard basis vector. To avoid technicalities when dealing with
 generic initial ideals, we assume that $K$ is infinite.

Since in this paper all $\ZZ^n$-graded 
$S$-modules are ideals in the ring $S$ or quotients of $S$, from now 
on we assume that $\ZZ^n$-graded modules $M$
are finitely generated and
$M = \bigoplus_{a \in \NN^n} M_a$ 
has non-zero components $M_a$ concentrated in
non-negative degrees. For $A\subseteq [n]$ we write
$M_A$ for the component
$M_{\sum_{j\in A} e_j}$ and say that an element $f \in M_A$
has $\ZZ^n$-degree $A$.

The group $G=\GL_{m_1}(K)\times\cdots\times\GL_{m_n}(K)$ acts  on $S$ as a group of $\ZZ^n$-graded $K$-algebra automorphisms with $\GL_{m_j}(K)$ acting on the variables
$x_{1j},\ldots, x_{m_jj}$ for $1 \leq j \leq n$.  
Let $>$ be a term order on $S$. All term orders in this paper are assumed to satisfy  $x_{1j}>x_{2j}> \cdots >x_{m_j j}$ for all $j$.

As in the standard $\ZZ$-graded situation, since $K$ is infinite there exists a non-empty Zariski open $U\subseteq G$ such that $\ini_>(gI)=\ini_>(g^\prime I)$ for all $g,g^\prime\in U$, where $\ini_>(I)$ is the 
initial ideal of $I$ with respect to $>$. 
We call such a $g$ a generic $\ZZ^n$-graded change of coordinates. We refer the   reader to~\cite[Theorem~15.23]{E} for details on generic initial ideals in the $\ZZ$-graded case and to~\cite[Section~1]{ACD} for a similar discussion in the $\ZZ^n$-graded case.
This leads to the definition of the $\ZZ^n$-graded generic initial ideal. 

\begin{definition}
The $\ZZ^n$-graded generic initial ideal $\gin_>(I)$ of $I$ with respect to $>$ is the ideal 
$\ini_>(gI)$, where $g$ is a generic $\ZZ^n$-graded coordinate change.  
\end{definition}

Let $B=B_{m_1}(K)\times \cdots  \times B_{m_n}(K)$ be the Borel subgroup of $G$ of upper triangular matrices in $G$. 

It turns out that $\gin(I)$ is fixed by the action of $B$.
Ideals fixed by $B$ are called Borel-fixed ideals. Let $J$ be a monomial ideal. 
If $\text{char}(k) = 0$, or if $J$ radical, the following condition (BF) is known to be equivalent to being 
Borel-fixed. 

\begin{align*}
    (\text{BF}) & \text{ for all monomials } x_{ij}w \in J \text{ one has }
    x_{kj}w\in J \text{ for all } k<i.
\end{align*}

Note that, if instead of our convention
$x_{1j}>x_{2j}> \cdots > x_{m_j\,j}$
the variables in each block are ordered differently then $\gin(I)$ is fixed by
a conjugate of $B$ and the condition (BF) changes accordingly.

Cartwright-Sturmfels ideals were introduced in \cite{CDG2} and their properties were
further developed in \cite{CDG1,CDG3,CDG4}.  Radical Borel-fixed ideals play an important role in this theory. 

\begin{definition}  
We let $\Brad(S)$  be the set of radical monomial ideals of $S$ which are
Borel-fixed.
\end{definition} 

Before we give the definition of Cartwright-Sturmfels ideals, we have to 
introduce the $\ZZ^n$-graded Hilbert series.
For a $\ZZ^n$-graded $S$-module $M=\oplus_{a\in \NN^n} M_a$ the $\ZZ^n$-graded Hilbert series of $M$ is defined as the formal
 power series  
$$\HS_M(Z_1,\dots,Z_n)=\sum_{a\in \NN^n} (\dim_K M_a)  Z^a\in\QQ[[Z_1,\dots,Z_n]].$$ 

\begin{definition}
\label{sameHS}
A $\ZZ^n$-graded ideal $I$ of $S$ is Cartwright-Sturmfels, or CS for short, if there exists  $J\in\Brad(S)$  which has the same 
$\ZZ^n$-graded Hilbert series as $I$. 
\end{definition}

By \cite[Theorem 3.5]{CDG1} the ideal $J$ that appears in Definition \ref{sameHS} is uniquely determined by (the $\ZZ^n$-graded Hilbert series of) $I$.  Furthermore it is proved in \cite[Corollary 3.7]{CDG1} that 
$J = \gin(I)$ where the gin is computed with respect to any term order satisfying our constraint on the order of variables. 
This leads to the following characterization, see \cite[Proposition 2.6]{CDG2}.

\begin{proposition}\label{ginRad}
A $\ZZ^n$-graded ideal $I$ is Cartwright-Sturmfels if and only if $\gin(I) \in \Brad(S)$.
\end{proposition} 

In \cite[Theorem 1.16]{CDG2} it is shown that the class of CS ideals is closed under some natural operations. In particular Theorem 1.16(5) states the following:
\begin{proposition}
\label{sections}
Let $L$ be a $\ZZ^n$-graded linear form in $S$, and identify $S/(L)$  with a polynomial ring with the induced $\ZZ^n$-graded structure.
If $I$ is a CS ideal of $S$ then $I+(L)/(L)$ is a CS ideal
in $S/(L)$.
\end{proposition}

We end the section with two results from \cite{Radical:support}, that will be useful in the following.

Consider the polynomial ring $T = K[y_1,\dots, y_n]$ with  
$\ZZ^n$-grading $\deg y_{j}=e_j$ for $j\in [n]$.  
Let $\M(T,m)$ be the set of the monomial ideals of $T$ generated by monomials $y_1^{a_1}\cdots y_n^{a_n}$, whose exponent vector is bounded from above by $\mm=(m_1,\dots,m_n)$,
i.e.,  $a_j \leq m_j$ for $j \in [n]$. 
By identifying $y_j$ with $x_{j1} \in S$  we can consider
$T$ as a subring of $S$. Extending this convention,
for a monomial $y_1^{a_1}\cdots y_n^{a_n}$ with exponent vector
bounded by $\mm$ we consider its polarization 
$x_{11}\cdots x_{1a_1} \cdots x_{n1}\cdots x_{na_n}$
as a monomial in $S$. Via this identification, the 
polarization $\pol(I)$ of any $I\in \M(T,\mm)$ can be seen as
an ideal of $S$. 

In \cite{Radical:support} it is proved that the  $\psi: \M(T,\mm) \to  \Brad(S)$ sending $I \in \M(T,\mm)$ to $\psi(I)=J=\pol(I)^*$, where 
  $^*$ denotes the Alexander dual, is a bijection between $\M(T,\mm)$ and  $\Brad(S)$

For a $\ZZ^n$-graded $S$-module $M=\oplus_{a\in \NN^n} M_a$ the $\ZZ^n$-graded Hilbert series of $M$ 
has a rational expression  
 $$\HS_M(Z)=\frac{\cK_M(Z_1,\dots,Z_n)}{\prod_{i=1}^n(1-Z_i)^{m_i}}.$$
   The polynomial $\cK_M(Z)=\cK_M(Z_1,\dots,Z_n)$ has coefficients in $\ZZ$ and is known as the $\cK$-polynomial of $M$.    

\begin{proposition}\cite[Proposition 2.4]{Radical:support}
\label{JandE} 
 Let $J\in \Brad(S)$ and $I\in \M(T,\mm)$. Then
 $\psi(I) = J$ if and only if $\cK_{S/J}(1-Z)=\cK_I(Z)$.  
\end{proposition} 
 
Consider now a sequence $\cA = A_1 , \ldots,  A_r$ of subsets of $[n]$. We define a
labeled undirected graph $G(\cA)$ with vertex set $\cA$. 
For $1 \leq \ell < k \leq r$ and every $j \in A_\ell \cap A_k$ we put 
an edge labeled by $j$ between $A_\ell$ and $A_k$. Note that, there may be parallel
edges in $G(\cA)$ but no loops.
By a cycle in $G(\cA)$ we mean a cycle in the
usual graph theoretic sense. 

We will use the following result which is part of the more comprehensive 
Theorem 4.1 \cite{Radical:support}, our conditions
 are conditions (4), (6)  in  Theorem 4.1.

\begin{proposition} \label{prop:456}
Given a sequence $\cA = A_1 , \ldots, A_r$ of non-empty subsets of $[n]$ 
the following are equivalent:
\begin{itemize}
\item[(1)] The graph $G(\cA)$ has no cycles with non-constant edge labels.
\item[(2)] There exists a field $K$ and numbers $m_1 , \ldots, m_n$ and a regular sequence $f_1 ,\ldots, f_r$ of polynomials $f_\ell$ of degree $A_\ell$ in $S=K[x_{ij}\mid 1\leq j\leq n,\ 1\leq i\leq m_j]$  such that the ideal $(f_1 ,\ldots , f_r)$ is CS.
\end{itemize}
\end{proposition}

For our purposes we recall that if $I$ is a CS ideal then all its initial ideals  (in every system  of $\ZZ^n$-graded coordinates) are  generated by elements of $S_A$ for various values of $A\subseteq [n]$. In other words, the generators of the initial ideal of a CS ideal involve at most one variable of degree $e_i$ for each $i\in [n]$.

\section{A base change argument}

In this section we prove Corollary \ref{cor:torvanishing}, which will
be used in Section \ref{section:NewCS} for proving the CS property for
one of our ideal constructions. The corollary is deduced from 
Proposition \ref{SriL}. Since its proof would not benefit from 
specializations, the proposition is formulated in a rather general form.

\begin{proposition} 
\label{SriL}
    Let $R$ be a commutative ring and let $A,B$ be flat algebras over $R$.
 Set $C=A\otimes_R B$, and for a right $A$-module $M$ and a left $B$-module $N$, consider $M\otimes_R B$ as a right $C$-module and $A\otimes_R N$ as a left $C$-module. There is a natural isomorphism
\[
\mathrm{Tor}^R_i(M,N) \cong \mathrm{Tor}^C_i(M\otimes_R B, A\otimes_R N) \text{ for all $i$.}
\]
\end{proposition}

\begin{proof}
 Let $\fF \to M$ and $\fG \to N$ be free resolutions of $M$ and $N$ over $A$ and $B$, respectively. 
 Since $A$ and $B$ are flat over $R$, we get that the induced maps
 \[
 \fF\otimes_R B\longrightarrow M\otimes_R B \quad\text{ and }\quad A\otimes_R \fG \longrightarrow A\otimes_R N
 \]
 are free resolutions over $C$. We claim that the natural map

 \begin{center}
 $\begin{array}{ccc}
 \fF \otimes_R \fG & \longrightarrow & (\fF \otimes_R B) \otimes_{C} (A \otimes_R \fG) \\
 (f\otimes g) &  \mapsto & (f\otimes 1)\otimes (1\otimes g)
 \end{array}$
 \end{center}
 
 \noindent is $C$-linear and an isomorphism of complexes. The $C$-linearity is a direct verification. To see that it is an isomorphism, note that since $\fF$ and $\fG$ are free over $A$ and $B$, respectively, it suffices to check that map is an isomorphism when $\fF=A$ and $\fG=B$, and that is clear. Note, that $\fF\to M$ and $\fG\to N$ can be regarded as flat resolutions of $M$ and $N$ as $R$-modules. Thus  one gets isomorphisms
 \begin{align*}
 \mathrm{Tor}^R_i(M,N) 
    &\cong \mathrm{H}_i(\,\fF \otimes_R \fG\,) \\    
    &\cong \mathrm{H}_i\big(\,(\fF \otimes_R B) \otimes_{A\otimes_R B} (A\otimes_R \fG)\,\big) \\
    &\cong \mathrm{Tor}^C_i(\,M\otimes_R B, A\otimes_R B\,)
 \end{align*}
This proves the assertion. 
\end{proof}

As a corollary we have the following result which will become
crucial in the proof of Proposition \ref{prop:linked}. 

\begin{corollary} \label{cor:torvanishing}
Suppose $S=K[\x,\y,\z]$ where $\x,\y,\z$ are disjoint sets of variables with  ${\x}=x_1,\dots, x_p$. 
Let  $I$ and $J$ be ideals of $S$ such that $I$ is extended from $K[\x,\y]$ and $J$ is extended from $K[\x,\z]$. Then 

\[
\mathrm{Tor}^S_i(S/I,S/J)=0 
\]
for all $i>p$. 
In particular, if $p = 1$ then the tensor product of  
free resolutions of $I$ and $J$ yields a free resolution of $IJ$.
\end{corollary}

\begin{proof}
Set  $R=K[{\x}],\ A=K[{\x,\y}],\ B=K[{\x,\z}]$ so that $A\otimes_R B=S$ and $A$ and $B$ are flat (actually free) over $R$. Let $I_1$ be the ideal of $A$ with $I_1S=I$ and $J_1$ be the ideal of $B$ with $J_1S=J$. Set $M=A/I_1$ and $N=B/J_1$. We have $M\otimes_R B=S/I$ and , $A\otimes_R N=S/J$. Applying Proposition \ref{SriL} to the present setting we have  

\begin{align*}
\mathrm{Tor}^R_i(A/I_1,B/J_1) & \cong
    \mathrm{Tor}^S_i(S/I,S/J)
 \end{align*}

and $\mathrm{Tor}^R_i(A/I_1,B/J_1)$ vanishes for $i$ larger than the global dimension of $R$ that, by the general version of  Hilbert's syzygy theorem \cite[Theorem 8.37]{Rotman}, is $p$.  

Finally when $p=1$ we have 
$$
\mathrm{Tor}^S_i(S/I, S/J) = 0 \quad \text{ for all } i > 1
$$
hence 
$$
\mathrm{Tor}^S_i(I, J)=\mathrm{Tor}^S_{i+2}(S/I, S/J) = 0  \text{ for all } i > 0.$$
This implies that the tensor product of free resolutions of $ I $ and $ J $ gives a free resolution of $ I \otimes_S J$.
 Now we observe that the kernel of the natural map   $ I \otimes_S J \to IJ$ gets identified with 
 $\mathrm{Tor}^S_1(S/I, J)=\mathrm{Tor}^S_2(S/I, S/J)$ which is $0$. 
That is, $I \otimes_S J \simeq IJ$.
\end{proof}


\section{Constructions preserving the Cartwright-Sturmfels property}
\label{section:NewCS}

First, we show that the CS property is inherited by ideals generated by
all elements of a fixed squarefree $\ZZ^n$-degree.

\begin{lemma} \label{easy1} 
Let $I\subset S$ be a CS ideal and $A\subseteq [n]$.  Let $J$ be the ideal of $S$ generated by the vector space $I_A$.  Then $J$ is CS and $\ini_{>}(J)$ is generated by the monomial vector space $\ini_{>}(I_A)$ for every  term order $>$ of $S$. 
\end{lemma}

\begin{proof}
  Let $u$ be a monomial from the minimal monomial 
  generating set of an initial ideal of $J$.  
  By definition $J$ is generated by polynomials which are linear combinations of monomials of
  $\ZZ^n$-degree $A$. It follows that $\deg u=\sum_{i\in A} a_i\,e_i $ for integers $a_i >0$. Since $J\subset I$ and $I$ is CS  
  there exists a subset $B$ of $A$ and $f\in I_B$ such that $\ini_{>}(f) | u$. Let $v$ be a monomial of degree $A\setminus B$
  such that $v|u$.
  Then $v\,f\in I_A$ and $\ini_{>}(v\,f)=v\,\ini_{>}(f) | u$. 
\end{proof} 

We use the preceding lemma to give a criterion when the CS property is
inherited by ideals which are generated by all polynomials from a set of squarefree
$\ZZ^n$-degrees in a CS ideal.

\begin{proposition}
\label{easy2} Let $\mathcal{F}\subseteq 2^{[n]}$ such that:
\begin{align} \label{eq:union}
    \text{ if } A,B\in \mathcal{F}  \text{ and } A\cap B\neq \emptyset \text{ then } A\cup B\in \mathcal{F}. 
    \end{align}
    Let $I$ be a CS ideal and  let $J$ be the ideal generated by the vector spaces $I_A$ with $A\in \mathcal{F}$. Then $J$ is CS. More precisely, for every term order $>$ the initial ideal  
$\ini_{>}(J)$ is generated  by the monomial vector spaces $\ini_{>}(I_A)$ with $A\in \mathcal{F}$. 
\end{proposition}

\begin{proof} Let $>$ be a term order. Let $f\in J$ be of 
$\ZZ^n$-degree
$a = \sum_{i=1}^ n a_i\, e_i\in \ZZ^n$ and set $E_f=\{\, i\, :\, a_i>0\,\}$. Let $A_1,\dots,A_s$ be the inclusionwise maximal elements of $\{ \,B \in \mathcal F\,:\,B \subseteq E_f\,\}$ and let $U_i$ be the ideal generated by the vector space $I_{A_i}$ for $1 \leq i \leq s$. 
 It follows that $f\in \sum_{i=1}^s U_i$. If $A_i \cap A_j \neq \emptyset$ for some $1 \leq i < j \leq s$ then by \eqref{eq:union} $A_i \cup A_j \in \mathcal F$
 and $A_i \cup A_j \subseteq E_f$ contradicting their maximality. 
 Hence the $A_1,\ldots, A_s$ are pairwise disjoint. 
 As a consequence, the ideals $U_i$ are generated by polynomials in disjoint sets of variables. The Buchberger criterion then implies
 that $\ini_{>}(f)\in \ini_{>}(U_i)$ for some $i$. By Lemma \ref{easy1} we have that $\ini_{>}(U_i)$ is generated by the vector space $\ini_{>}\big(\,I_{A_i}\,\big)$. We conclude that $\ini_{>}(f)$ belongs to the ideal generated by $\ini_{>}(I_{A_i})$. Hence $\ini_{>}(J)$ is radical. Since this holds also in generic coordinates, 
 the proof is complete. 
\end{proof}   

For a $\ZZ^n$-ideal $J$ generated in $\ZZ^n$ degrees
$\leq e_1+\ldots +e_\ell$ and a $\ZZ^n$-graded  ideal 
$H$ generated in degrees
$\leq e_{\ell+1}+\cdots +e_n$ it is easy to see that 
if $J$ and $H$ are CS then so is $J+H$. 
Indeed, the following proposition shows that a mild overlap
of degrees still allows the same conclusion.

  \begin{proposition} 
   \label{prop:linked} 
   Let $I$ and $J$ be CS ideals 
  such that for some $1 \leq \ell \leq n$ the ideal $I$ 
  is generated in degrees $\leq e_1 + \cdots + e_\ell$ 
  and $J$ is generated in degrees $\leq e_\ell + \cdots + e_n$. 
  Then $I+J$ is CS.
  \end{proposition}
\begin{proof}
We first reduce to the case when there are minimal generating sets of $I$ and $J$ that use different sets of variables. 

For the reduction we consider the polynomial ring $S'$ which has the same set of variables as $S$ except that for each variable $x$ 
of degree $e_\ell$ we have a duplicate $x'$ also of degree $e_\ell$. In $S'$ we consider
the ideal $I'$ generated by the generators of
$I$ and $J'$ generated be the generators of $J$ where in each generator the variables
$x$ of degree $e_\ell$ have been replaced by
their duplicates $x'$. It follows that 
$I'$ is generated in $\ZZ^n$-degree $\leq e_1+\cdots +e_\ell$, $J'$ is generated in 
$\ZZ^ n$-degrees $\leq e_\ell+\cdots +e_n$ and
$I'$ and $J'$ are generated by polynomials in 
disjoint sets of variables.

Since $I$ and $J$ are
CS in $S$ so are $I'$ and $J'$ in $S'$. 
The ideal $I+J$ in $S$ arises from 
specializing $x'$ to $x$ for any variable $x'$
in $S'$ of degree $e_\ell$. Since by Proposition \ref{sections} the CS property is preserved under $\ZZ^n$-graded specializations, the
assertions follows if $I'+J'$ is CS.

From now on we can assume that $I$ and $J$ 
are generated by polynomials in disjoint sets of variables. It follows that a minimal free
resolution of
$S/I+J$ can be obtained from the tensor product of a minimal free resolution of $S/I$ with a minimal free resolution of $S/J$. Hence the $\cK$-polynomial of $S/I+J$ is the product $\cK_{S/I}(Z)\,\cK_ {S/J}(Z)$.  
Therefore, by Proposition \ref{JandE} it is enough to prove that $\cK_{S/I}(1-Z)\,\cK_{S/J}(1-Z)$  is the $\cK$-polynomial of a monomial ideal in  $T=K[y_1,\ldots,y_n]$, that is, of an ideal in $\M(T,\mm)$.
Since $I$ and $J$ are CS we know that there are monomial ideals $I'$ and $J'$ in  $T$ such that $\cK_{S/I}(1-Z)$ is the $\cK$-polynomial of $I'$  and  $\cK_{S/J}(1-Z)$ is the 
$\cK$-polynomial of $J'$. 
Furthermore, by the assumption on $I$ and $J$, we have that at most the variable $y_\ell$
appears in the minimal monomial generating sets of both $I'$ and $J'$.  By Corollary \ref{cor:torvanishing}
we have that $\mathrm{Tor}_i^T(T/I',T/J')=0$ for $i>1$ 
and a minimal free resolution of the product of  $I'J'$ can be obtained by the tensor product of a minimal free
resolution of $I'$ with a minimal free resolution of $J'$. It follows that the $\cK$-polynomial of $I'J'$ is $\cK_{S/I}(1-Z)\,\cK_{S/J}(1-Z)$.
\end{proof}
   
   Proposition \ref{prop:obstruct} for $\ell = 2$ and $t \geq 3$ (see also
Figure \ref{fig:obs}) shows that the conditions of the Proposition \ref{prop:linked} cannot be weakened to allow 
overlaps of squarefree $\ZZ^ n$-degrees of size $2$ or larger. 


\section{Generalized binomial edge ideals and higher-order minors}
\label{Gen}
In this section, we specialize to the situation when 
 $S=K[X]$, with $X$ a generic matrix of size $m\times n$
 with the $\ZZ^n$-graded structure given by $\deg x_{ij} = e_j$ for every $j=1,\ldots, n$. We recall that our standard assumption on the term order is that 
 $x_{ij}>x_{kj}$ if $i<k$. 
 As mentioned before, the ideal $I_2(X)$ generated by the 
 $2$-minors of $X$ is known
 to be CS, see \cite[Theorem 2.1]{CS}. In particular, for every $A\subseteq  [n]$
 the degree $A$ component of the $\ZZ^n$-graded generic initial ideal of $I_2(X)$ is generated by the monomials

\begin{equation} 
\label{gin}
 \prod_{j\in A} x_{i_jj} 
 \end{equation}  
 
 such that 
 
 \begin{equation} 
\label{cond} 
   i_j\in [m]   \mbox{ for every } j\in A \mbox{  and  } \sum_{j\in A} i_j\leq m(|A|-1).
  \end{equation}   
 
 In a next step we confine the generating set of the ideal to
 minors of columns of $X$ selected from the edge set of a graph. 
 More precisely, an (undirected simple) graph $G = ([n],E)$ is a pair 
 of a vertex set $[n]$ and an edge set $E$ consisting of $2$-element
 subsets of $[n]$. 
 
 The generalized binomial edge ideal $I_G(m)$ is then defined as the 
 sum of the ideals generated by the $2$-minors of the columns
 $j,k$ of $X$ as $\{j,k\}$ varies in $E$.  Note that $I_G(2)$ is the (ordinary) binomial edge ideal which is CS by virtue of \cite{CDG4,CDG5}. 
 Rauh proved in \cite{Rauh} that $I_G(m)$ is radical for every $m$ by showing that it has a square-free initial ideal. 
 
 For $A \subseteq [n]$ we write
 $G_A$ for the graph $(A,E_A)$ induced on $A$ with edge set 
 $E_A = \{ e \in E~|~e \subseteq A\}$.
 We prove that: 
 
 \begin{theorem} 
 \label{main} Let $G = ([n],E)$ be a graph. 
 Then the generalized binomial edge ideal $I_G(m)$ is CS for all $m$. 
 More precisely, let $>$ be a term order satisfying
 $$x_{1j}>x_{2j} > \cdots > x_{mj} \text{ for } j \in [n].$$
 
 Then  the $\ZZ^n$-graded  generic initial ideal of $I_G(m)$ is generated by the monomials \eqref{gin} for the subsets $A \subseteq[n]$
 such that $G_A$ is connected and \eqref{cond} is satisfied.
\end{theorem}
\begin{proof} 
Set   $\mathcal F=\big\{\, A\subseteq [n]\, :\, G_A \text{ is connected }\,\}$ and  $I=I_2(X)$.  
If $A,B \in \mathcal F$ are
such that $A \cap B \neq \emptyset$ then $G_A$ and $G_B$ are two
connected graphs sharing at least one vertex. Thus 
$G_A \cup G_B = (A \cup B, E_A \cup E_B)$ is connected. Since
$E_A \cup E_B \subseteq E_{A \cup B}$ it follows that
$G_{A \cup B} = (A \cup B,E_{A\cup B})$ is connected and
$A \cup B \in \mathcal F$. In particular, $\mathcal F$
satisfies the assumptions of Proposition \ref{easy2}.

Let $J$ be the ideal generated by the vector spaces $I_A$ with $A\in \mathcal F$. By Proposition \ref{easy2} we know that $J$ is CS and 
its generic initial ideal is generated by the monomials \eqref{gin}
which satisfy \eqref{cond}. Therefore, it is enough to prove that $J=I_G(m)$. 

For the inclusion $J\supseteq I_G(m)$ observe that every edge $\{j,k\}$ of $G$ we have  $\{j,k\}\in \mathcal F$ and hence the generators of $I_G(m)$ are contained in $J$.  

For the other inclusion, we first observe the following. 
Denote by $C_{jk}$ the vector space  generated by the $2$-minors of $X$ using only columns $j$ and $k$ and by $C_\ell$ the vector space generated by the entries of column $\ell$. By double Laplace expansion one sees immediately that 

\begin{equation}
\label{expansion} 
 C_{jk} C_{\ell} \subseteq C_{\ell k} C_{j}  +  C_{j \ell} C_{k} .
\end{equation} 

Now let $A\in \FF$. We have to show that $I_A\subset I_G(m)$.   Since  $I_A$ is the sum of the vector spaces 
$$W_{j,k}:=C_{jk}\prod_{\ell\in A\setminus\{j,k\} }C_\ell$$ 
where $ \{j,k\}\subset A$, it is enough to show that $W_{j,k}\subseteq I_G(m)$.  As $G_A$ is connected, there exists a path $j=p_0,p_1,\dots, p_s=k$ in $G$ with $p_0,\ldots,p_s\in A$. 
We prove $W_{j,k} \subseteq I_G(m)$ by induction on $s$. 
If $s=1$ then $C_{jk} \subset I_G(m)$ and so $W_{jk} \subset  I_G(m)$. If $s>1$ then
by \eqref{expansion} we have
$C_{jk}C_{p_1}\subseteq C_{jp_1}C_{k}+C_{p_1k}C_{a}$. 
Thus  
$W_{jk} \subset W_{j,p_1}+W_{p_1,k}$. By induction on $s$, we may assume that both $W_{j,p_1}$ and $W_{p_1,k}$ are contained in $I_G(m)$ so that we can conclude that 
$W_{j,k}\subset   I_G(m)$. 
  \end{proof}  

In a next step we turn to higher-order minors. 
Let $s \geq 2$. An $s$-uniform hypergraph $H = ([n],E)$ is given by a vertex set $[n]$
and an edge set $E$ consisting of $s$-element subsets of $[n]$.
Note that $2$-uniform hypergraphs and graphs are the same concept.
Analogous to the graph case, for $A \subseteq [n]$, we write $H_A=(A,E_A)$ for the induced hypergraph on vertex set $A$
with edge set $E_A = \{\,e\in E\,|\, e \subseteq [n]\,\}$.
By $I_H(m)$ we denote the ideal of $s$-minors of $X$ with
columns $1 \leq j_1 < \cdots < j_s \leq n$ such that
$\{j_1,\ldots, j_s\} \in E$. 

If $E$ is the set of all $s$-subsets of $[n]$ then it was shown
in \cite[Theorem 2.2, Theorem 3.1]{CDG3} that $I_H(m)$ is CS if $s = m$ and $m \leq n$ or if $n=s$ and $s \leq m$; i.e., the cases where $I_H(m)$  is an ideal of maximal minors.

On the other hand for $s > 2$ a statement analogous to Theorem \ref{main} is false and a classification
of the hypergraphs $H$ for which $I_H(m)$ is CS seems to be out of reach. The main obstacle why for $s > 2$ we cannot pursue the
same strategy as in Theorem \ref{main} is the lack of a proper
replacement of \eqref{expansion} for higher minors.

We can still prove positive results, consider for example the following 
hypergraphs. The hypergraph $([n],\binom{[n]}{s})$ whose edge set
consists of all $s$-element subsets of $[n]$ is called the 
complete $s$-uniform hypergraph. 
We say a hypergraph $H = ([n],E)$ is a forest of complete $s$-uniform 
hypergraphs if either $H$ is a complete 
$s$-uniform hypergraph or $E = E_1 \cup E_2$, where

\begin{itemize}
    \item[(F1)] $E_1 \subseteq \binom{A}{s}$ and $E_2 = \binom{B}{s}$,
    \item[(F2)] $A \cup B = [n]$ and $\# (A \cap B) \leq 1$,
    \item[(F3)]  $(A,E_1)$ is a forest of $s$-uniform hypergraphs.
\end{itemize}

Note that the hypergraph $(B,E_2)$ in the preceding definition is a complete $s$-uniform hypergraph.

\begin{proposition} \label{prop:snaketrees}
Let $H$ be a forest of complete $s$-uniform hypergraphs.
Then $I_H(s)$ is CS.
  \end{proposition}
  \begin{proof} 
    We proceed by induction on the structure of the forest.
    
    If $H$ is a complete $s$-uniform hypergraph then 
    $I_H(s)$ is CS by \cite[Theorem 3.1]{CDG3}. 
    
    Now assume that 
    $H = ([n],E)$ and (F1), (F2) and (F3) are 
    satisfied. By induction we can assume that 
    for $H_A = (A,E_1)$ we have that $I_{H_A}(s)$ is
    CS. Again by \cite[ Theorem 3.1]{CDG3} it follows that 
    $I_{H_B}(s)$ is CS for $H_B = (B,E_2)$. 
    Since by (F2) we have $\#(A \cap B) \leq 1$ it follows 
    that $I = I_{H_A}(s)$ and $J = I_{H_B}(s)$ satisfy the
    assumptions of Proposition \ref{prop:linked}, thus $I_{H_A}(s) + I_{H_B}(s) = I_H(s)$ is
    CS.
\end{proof}

Indeed the arguments used in the proof allow a further
generalization of Proposition \ref{prop:snaketrees} to 
some type of forest of CS ideals. 

For that one can use the fact, proved in \cite[Theorem 2.2, Theorem 3.1]{CDG1},  
that the ideal of maximal minors is CS with respect to
row and column grading. 
The following proposition classifies the minor ideals
which are CS with respect to either gradings.

\begin{lemma} \label{lem:colrow}
    Let $H = ([n],E)$ be the complete $s$-uniform hypergraph.
    Then the following are equivalent:
\begin{itemize}
    \item[(i)] $I_H(m) $ is CS for the column grading.
    \item[(ii)] $I_H(m) $ is CS for the row grading.
    \item[(iii)] $s =1$ or $s =2$ or $s = \min (m,n)$.
    \end{itemize}
\end{lemma}
\begin{proof}
The fact that (iii) implies (i) and (ii) is proved in
\cite{CDG1}. 
The implications (i) $\Rightarrow$ (iii) and 
(ii) $\Rightarrow$ (iii) are trivial for $s =1$ and 
follow from Theorem 4.6.17 in
\cite{BCRV}. Indeed the theorem shows that for $s \geq 2$ only in the cases listed in (iii) the $\ZZ^n$-degree is multiplicity free. However, that condition is necessary for an ideal to be CS \cite[Proposition 2.6]{CDG4}.
\end{proof}

Note, that in general ideals of minors which are CS with respect to the column grading do not have to be CS with respect to the row grading and vice versa. 
For example consider the ideal $I_G(2)$ of $2$-minors
for the graph $G = ([4],\{\,\{1,2\},\{3,4\}\,\})$.
By Theorem \ref{main} $I_G(2)$ is CS with respect to the
column grading. Since the two generators 
of $I_G(2)$ use different sets of variables, $I_G(2)$ is
a complete intersection. With the row grading both 
generators have degree $e_1+e_2$. But the sequence 
$\cA$ of sets
$A_1 = A_2 = \{1,2\}$ defines a graph $G(\cA)$ with two
vertices and two edges labeled $1$ and $2$ between the two 
vertices. This allows for cycles with non-constant edges labels. Hence by Proposition \ref{prop:456} it follows
that $I_G(2)$ is not CS with respect to the row labeling.

Now we can provide an example which indicates how far
one can push the proof method of
Proposition \ref{prop:snaketrees}. 
We confine our discussion of the extension to an example
since the actual definitions would be rather technical.

\begin{example}

Let $X$ be a generic $10 \times 14$ matrix of variables. 
Consider the following set of ideals (see also Figure \ref{fig:minors}):
\begin{itemize}
    \item $J_1$ the ideal of $3$-minors from 
rows $1$ to $3$ and columns $1$ to $4$, 
    \item $J_2$ the ideal of $3$-minors. 
    from rows $2$ to $5$ and columns $4$ to $6$,
    \item $J_3$ the ideal of $4$-minors from 
    rows $5$ to $8$ and columns $6$ to $10$,
    \item $J_4$ the ideal of $2$-minors from 
    rows $9$ and $10$ and columns $11$ to $14$.
    \end{itemize}
    Each ideal is an ideal of maximal minors and 
    hence by Lemma \ref{lem:colrow} is CS in the set of
    variables from the respective rows and columns. 
    Since we use column grading adding more rows does
    not change the $\cK$-polynomial. 
    Therefore, by definition each of them is CS as an ideal in  the polynomial ring $K[X]$. 
    Since the column indices for the four ideals are either pairwise disjoint or  share exactly one element, the
    arguments used in the proof of Proposition \ref{prop:snaketrees} apply and show that 
    $J_1 + J_2 + J_3 + J_4$
    is CS.

Note that for the example only the intersection property of the column indices is needed. Also the row and column sets 
of the maximal minor ideals do not have to be consecutive. Indeed, the argument allows for arbitrary sums of ideals 
$J_1,\ldots, J_r$ satisfying:   
\begin{itemize}
    \item each $J_i$ is CS in the column grading  and 
    \item  
    if $A_i$ is the union of the support sets of the 
    multidegrees of the
    generators of $J_i$, then 
    $\#\big( A_i \cap \bigcup_{j \neq i} A_j \,\big) \leq 1$.
    \end{itemize}

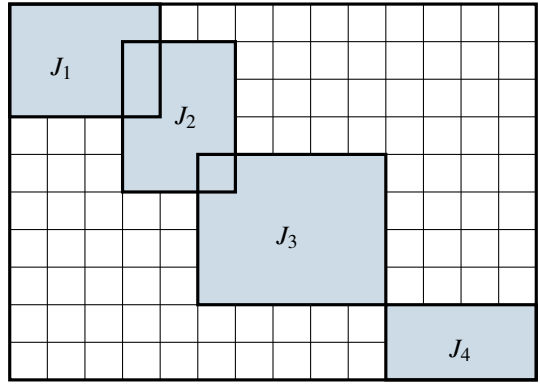
\begin{figure}[ht!]
\definecolor{airforceblue}{rgb}{0.33,0.53,0.67}
\begin{tikzpicture}[scale=0.5]

\foreach \i in {0,...,14} {
    \draw (\i,0) -- (\i,10);}
 \foreach \i in {0,...,10} { 
    \draw (0,\i) -- (14,\i);
}

\foreach \r in {7,8,9} 
\foreach \c in {0,1,2,3} 
    \fill[airforceblue!30] (\c,\r) rectangle ++(1,1);

\foreach \r in {5,6,7,8} 
\foreach \c in {3,4,5} 
    \fill[airforceblue!30] (\c,\r) rectangle ++(1,1);

\foreach \r in {2,3,4,5} 
\foreach \c in {5,6,7,8,9}
    \fill[airforceblue!30] (\c,\r) rectangle ++(1,1);

\foreach \r in {0,1} 
\foreach \c in {10,11,12,13}
    \fill[airforceblue!30] (\c,\r) rectangle ++(1,1);

\draw[very thick] (0,0) rectangle (14,10);
\draw[very thick] (0,10) rectangle (4,7);
\draw[very thick] (3,5) rectangle (6,9);
\draw[very thick] (5,2) rectangle (10,6);
\draw[very thick] (10,0) rectangle (14,2);
\node[label={45:$J_1$}] at (0.6,7.5) {};
\node[label={45:$J_2$}] at (3.9,6.2) {};
\node[label={45:$J_3$}] at (6.6,3) {};
\node[label={45:$J_4$}] at (11.2,0) {};
\end{tikzpicture}
\caption{Example of CS-minor ideal} \label{fig:minors}
\end{figure}
\end{example}

As we will demonstrate now, there are large classes of $m$-uniform hypergraphs $H = ([n],E)$ 
for which $I_H(m)$ is not CS. 
Given $j \in \NN$ we write $[j]_m$ for
 the set $\{j,j+1,\cdots,j+m-1\}$ with $0<j\leq n-m+1$. We have: 

\begin{lemma}
\label{complete}
  Let $1\leq j_1<j_2<\dots <j_t \leq n-m+1$ be such that $j_{\ell+1}-j_\ell>1$ for all $\ell=1,\dots, t-1$. 
  Let $H = ([n],E)$ be the $m$-uniform hypergraph with 
  edges $[j_\ell]_m$ for $1 \leq l \leq t$.
  Then $I_H(m)$ is a prime ideal and the
  defining generators form a regular sequence. 
 
\end{lemma}
\begin{proof}
 For $1 \leq \ell \leq t$ let $\Delta_{j_\ell}$ be the $m \times m$-minor of $X$ with column set $[j_\ell]_m$. 
 The minors $\Delta_{j_1}, \ldots ,
 \Delta_{j_t}$ have disjoint sets of
 variables on their respective main diagonals.
 It follows that, with respect to a diagonal term order, their initial terms are pairwise coprime. This implies that 
 $\Delta_{j_1}, \ldots \Delta_{j_t}$ is
 a regular sequence (see for example  \cite[Proposition 1.2.12]{BCRV}  or \cite[Problem 3.1, p.50]{monomialidealsbook}).  

 To prove that the ideal $I_H(m) =\big(\,\Delta_{j_1}, \Delta_{j_2},\dots, \Delta_{j_t}\,\big)$ is prime we consider the $(m-1)$-minor $W_b$ of $X$ with rows $1,2,\dots m-1$ and columns $b,b+1,\dots,b+m-2$. Let $f$ be the product $W_{j_1+1}
 \cdots W_{j_t+1}$. The assumption $j_{\ell+1}-j_\ell>1$ guaranties that the diagonal leading term of each $W_{j_{\ell'}+1}$ is coprime to 
 the diagonal leading terms of the $\Delta_{j_\ell}$ for $1 \leq \ell,\ell'\leq t$. Hence by the same arguments as above 
 $\Delta_{j_1}, \Delta_{j_2},\ldots, \Delta_{j_t},f$ is a regular sequence as well.  In particular, $f$ is a non-zero divisor on $S/I$. This implies that $S/I$ embeds in $(S/I)_f$. Therefore, it is enough to prove that $(S/I)_f=S_f/IS_f$ is a domain. Since $f$ is invertible in $S_f$ every factor of $f$ is invertible too.
 In particular, $W_{j_\ell+1}$ is invertible for all $\ell$. 
Now by Laplace expansion 
$$\Delta_{j_\ell}=x_{mj_\ell}W_{j_\ell+1}+F_\ell$$ 
with $F_\ell$ a polynomial in the variables $x_{hk}$ for $k=j_\ell$ and $h<m$ or $k>j_\ell$. It follows that in $S_f$ we have 
$$W_{j_\ell+1}^{-1}\Delta_{j_\ell}=x_{mj_\ell}+W_{j_\ell+1}^{-1}F_\ell$$ 
and 
$$IS_f=\big(\,x_{mj_\ell}+W_{j_\ell+1}^{-1}F_\ell : \ell=1,\dots,t \,\big).$$
This shows that $S_f/IS_f$ is isomorphic to a localization of a polynomial ring over a field, hence a domain.  
\end{proof}

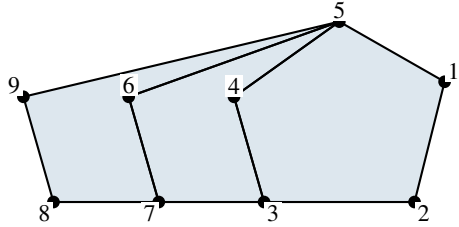
\begin{figure}[ht!]
    \centering
    \definecolor{airforceblue}{rgb}{0.33,0.53,0.67}

 \begin{tikzpicture}[scale=2, every node/.style={fill=white, inner sep=1pt, font=\small}]

\coordinate (3) at (0,0);
\coordinate (2) at (1,0);
\coordinate (1) at (1.2,0.8);
\coordinate (5) at (0.5,1.2);
\coordinate (4) at (-0.2,0.7);

\coordinate (6) at ($(4)+(-0.7,0)$);
\coordinate (7) at ($(3)+(-0.7,0)$);

\coordinate (8) at ($(7)+(-0.7,0)$);
\coordinate (9) at ($(6)+(-0.7,0)$);

\fill[airforceblue!20]  (1)--(2)--(3)--(4)--(5)--cycle; 
\fill[airforceblue!20] (3)--(4)--(5)--(6)--(7)--cycle; 
\fill[airforceblue!20]  (5)--(6)--(7)--(8)--(9)--cycle; 

 \draw[thick] (1) -- (2) -- (3) -- (4) -- (5) -- (1);
 \draw[thick] (3) -- (4) -- (5) -- (6) -- (7) -- (3);
 \draw[thick] (5) -- (6) -- (7) -- (8) -- (9) -- (5);
   
  \fill (1) circle (1.2pt) node[above right] {$1$};
\fill (2) circle (1.2pt) node[below right] {$2$};
\fill (3) circle (1.2pt) node[below right] {$3$};
\fill (4) circle (1.2pt) node[above] {$4$};
\fill (5) circle (1.2pt) node[above] {$5$};
\fill (6) circle (1.2pt) node[above] {$6$};
\fill (7) circle (1.2pt) node[below left] {$7$};
\fill (8) circle (1.2pt) node[below left] {$8$};
\fill (9) circle (1.2pt) node[above left] {$9$};\end{tikzpicture}
  
\caption{Example of a $5$-uniform hypergraph satisfying the conditions of Lemma \ref{complete} for
$j_1 = 1$, $j_2 = 3$ and $j_3 = 5$} \label{fig:complete}

\end{figure}

Using the lemma we can derive a set of obstructions to 
the CS property.

\begin{proposition} \label{prop:obstruct} Let $m \geq 3$ and 
$t\in \NN$ be such that $(t+1)(m-1)+1\leq n$. Let $H = ([n],E)$ be the
$m$-uniform hypergraph with edges
$$[1]_m,[m]_m,[2m-1]_m,\ldots,[\ell(m-1)+2-m]_m,\ldots,[t(m-1)+2-m]_m,$$
and 
$$[(t+1)(m-1)+2-m]_m \setminus \{(t+1)(m-1)+1\} \cup \{1\}.$$
Then $I_H(m)$ is not CS. 
\end{proposition}
\begin{proof}

For $1 \leq \ell \leq t$ let $\Delta_{\ell(m-1)+2-m}$ be the $m$-minor of $X$ with column set $[\ell(m-1)+2-m]_m$ and let $W$ the $m$-minor of $X$ with 
column set $[(t+1)(m-1)+2-m]_m \setminus \{(t+1)(m-1)+1\} \cup \{1\}$.

By $m \geq 3$ the assumptions of 
Lemma \ref{complete} are satisfied for the $m$-uniform hypergraph $H' = ([n],E')$
with edges
$[1]_m,[m]_m,[2m-1]_m,\ldots,[\ell(m1-)+2-m]_m,\ldots,[t(m-1)+2-m]_m$. It follows that the ideal $$I_{H'}(m) =\big(\,\Delta_{1}, \Delta_{m}, \Delta_{2m-1}, \Delta_{3m-2}\dots, \Delta_{tm-t+1}\,\big)$$ is prime and its defining generators form a regular sequence. Observe, that $W\not\in I_{H'}(m)$ because $W$ is of the same total degree as the other generators but of different $\ZZ^n$-degree.  Hence the generators of $I_H(m)$ form are regular
sequence.

In our situation condition (1) 
of Proposition \ref{prop:456} 
is not satisfied by the support of the degrees $\cA$ of the generators of $I_H(m)$.
Indeed, $G(\cA)$ is actually a single 
cycle with distinct labels. 
Hence, Proposition \ref{prop:456} implies 
that $I_H(m)$ is not CS. 
\end{proof}

\begin{figure}
\definecolor{airforceblue}{rgb}{0.33,0.53,0.67}
\begin{tikzpicture}[scale=1.2]

\node at (1.5,2.5) {Proposition \ref{prop:obstruct}: $m=s=3$, $n = 4$};
\begin{scope}[shift={(0,0)}]
  \coordinate (p2) at (0,0);
  \coordinate (p1) at (2,0);
  \coordinate (p3) at (1,1.5);
  \coordinate (p4) at (3,1.5);
  
  \fill[airforceblue!20] (p1) -- (p2) -- (p3) -- cycle;
  \fill[airforceblue!20] (p1) -- (p3) -- (p4) -- cycle;
  
  \draw[thick] (p2) -- (p3) -- (p4) -- (p1);
  \draw[thick] (p2) -- (p1) -- (p3);
  
  \foreach \i/\p in {1/p1, 2/p2, 3/p3, 4/p4} {
    \filldraw (\p) circle (2pt) node[font=\small, anchor=center,
    xshift=-0.2cm,yshift=(int((\i-1)/2)*2-1)*0.2 cm] {\i};
  }
\end{scope}

\node at (7,2.5) {Proposition \ref{prop:obstruct}: $m=s=3$, $n = 8$};
\begin{scope}[shift={(7,0)}]

  \foreach \i in {1,...,8} {
    \coordinate (p\i) at ({90 - 45*(\i - 1)}:2cm);
  }

  \fill[airforceblue!20] (p1) -- (p2) -- (p3) -- cycle;
  \fill[airforceblue!20] (p3) -- (p4) -- (p5) -- cycle;
  \fill[airforceblue!20] (p5) -- (p6) -- (p7) -- cycle;
  \fill[airforceblue!20] (p1) -- (p7) -- (p8) -- cycle;

  \draw[thick] (p1) -- (p2) -- (p3) -- (p1);
  \draw[thick] (p3) -- (p4) -- (p5) -- (p3);
  \draw[thick] (p5) -- (p6) -- (p7) -- (p5);
  \draw[thick] (p1) -- (p7) -- (p8) -- (p1);

   \foreach \i in {1,2,3} {
    \filldraw (p\i) circle (2pt) node[font=\small, anchor=center,
    xshift=0.3cm,yshift=0.1 cm] {\i};}
  \foreach \i in {7,8} {
    \filldraw (p\i) circle (2pt) node[font=\small, anchor=center,
    xshift=-0.3cm,yshift=0.1 cm] {\i};}
     \foreach \i in {6} {
    \filldraw (p\i) circle (2pt) node[font=\small, anchor=center,
    xshift=-0.3cm,yshift=-0.1 cm] {\i};}
     \foreach \i in {4,5} {
    \filldraw (p\i) circle (2pt) node[font=\small, anchor=center,
    xshift=0.3cm,yshift=-0.1 cm] {\i};}
\end{scope}

\end{tikzpicture}

\bigskip

\centering

\begin{tikzpicture}[scale=1.2]

\node at (0,3) {Example \ref{ex:strange}:}; 

\begin{scope}[shift={(0,0)}]
  
  \foreach \i/\ang in {1/90, 3/-30, 5/-150} {
    \coordinate (p\i) at (\ang:2);
  };
 \foreach \i/\ang in {2/30, 4/-90, 6/150} {
    \coordinate (p\i) at (\ang:0.5);
  };

  \fill[airforceblue!20] (p1) -- (p2) -- (p3) -- cycle; 
  \fill[airforceblue!20] (p3) -- (p4) -- (p5) -- cycle; 
  \fill[airforceblue!20] (p1) -- (p5) -- (p6) -- cycle; 
  \fill[airforceblue!20] (p2) -- (p4) -- (p6) -- cycle;

  \draw[thick] (p1) -- (p2) -- (p3) -- cycle;
  \draw[thick] (p3) -- (p4) -- (p5) -- cycle;
  \draw[thick] (p1) -- (p5) -- (p6) -- cycle;
  \draw[thick] (p2) -- (p4) -- (p6) -- cycle;

  \foreach \i in {1,6} {
    \filldraw (p\i) circle (2pt)
      node[font=\small, anchor=south east, xshift=-0.05cm, yshift=0.05cm] {\i};
  }
  \foreach \i in {2,3} {
    \filldraw (p\i) circle (2pt)
      node[font=\small, anchor=south west, xshift=0.05cm, yshift=0.05cm] {\i};
  }
  \foreach \i in {4,5} {
    \filldraw (p\i) circle (2pt)
      node[font=\small, anchor=north, yshift=-0.12cm] {\i};
  }
\end{scope}
\end{tikzpicture}

\caption{Non CS Configurations from Proposition \ref{prop:obstruct} and Example \ref{ex:strange}}
\label{fig:obs}
\end{figure}

Unfortunately, the hypergraphs from  Proposition \ref{prop:obstruct} are not the
only obstructions for $I_H(m)$ to be CS.

\begin{example} \label{ex:strange}
Let $s =3$, $n= 6$ and $H = ([6],E)$ the 
$3$-uniform hypergraph with 
        $$E = \big\{ \,\{1,2,3\},\{3,4,5\},\{5,6,1\}, \{2,4,6\}\,\big\} \text{ (see Figure \ref{fig:obs})}.$$
        Then computer experiments show that, at least over a field of characteristic  $0$, the ideal $I_H(3)$ is not CS.
\end{example}

\section*{Acknowledgments}

We are grateful to Srikanth Iyengar for providing us with the arguments for
Proposition \ref{SriL}. Our original proof of Corollary \ref{cor:torvanishing} was inductive and less elegant.

\bibliographystyle{plainurl}

  \end{document}